\documentclass[12pt]{amsart}
\usepackage{amsmath}
\usepackage{amssymb}
\usepackage{amsfonts}
\usepackage{textcomp}
\usepackage{amsthm}
\usepackage{mathrsfs}
\usepackage[latin1]{inputenc}
\usepackage[all]{xy}
\usepackage{hyperref}
\usepackage{url}
\usepackage{graphicx}
\usepackage{tikz-cd}
\usepackage[paperheight=11in,paperwidth=8.5in,left=1.5in,top=1.25in,right=1.5in,bottom=1.25in,]{geometry}

\newcommand{\HH}{\mathbb{H}^3}
\newcommand{\RR}{\mathbb{R}}
\newcommand{\CC}{\mathbb{C}}
\newcommand{\PSL}{{\rm PSL}_2\mathbb{C}}
\newcommand{\SU}{{\rm SU}(2)}

\newcommand{\dd}{{\rm d}}
\newcommand{\idx}{{\rm Ind}}

\allowdisplaybreaks

\newtheorem{maintheorem}{Theorem}[section]

\newtheorem{maincor}{Corollary}[section]

\newtheorem{lem}{Lemma}[section]
\newtheorem{prop}{Proposition}[section]
\newtheorem{defi}{Definition}[section]
\newtheorem{cor}{Corollary}[section]

\newtheorem{remark}{Remark}[section]

\title{ On the topology and index of minimal/Bryant framed surfaces}
\author{Davi Maximo}
\address{Department of Mathematics, University of Pennsylvania, Philadelphia, PA 19104}
\email{dmaxim@math.upenn.edu}

\author{Franco Vargas Pallete}
\address{Department of Mathematics, Yale University, New Haven, CT 06511, USA}
\email{franco.vargaspallete@yale.edu}

\date{}
\thanks{The first author was supported by NSF grant DMS-1910496 and a Sloan Fellowship. The second author was partially supported by NSF grant DMS-2001997.}

\begin{document}

\maketitle
\begin{abstract}
    We study framed surfaces, which are a class of Euclidean minimal and hyperbolic CMC-1 surfaces that generalize immersed minimal surfaces in $\mathbb{R}^3$ and Bryant surfaces. For this class we prove a lower bound on the (unrestricted) Morse index by a linear function of the genus, number of ends and number of branch points (counting multiplicity), generalizing a result by Chodosh and the first author (\cite{ChodoshMaximo}). We include as well a description of the 1-to-1 correspondence between Euclidean minimal and Bryant surfaces, known in the literature as \emph{Lawson's correspondence}.
\end{abstract}

\section{Introduction}

Minimal surfaces are critical points for the area functional and their variational theory has been studied extensively. One of the fundamental questions about these objects is the relationship between their Morse index and their geometry and topology. 

The case of immersed minimal surfaces of $\Sigma \rightarrow \mathbb R^3$ is arguably the most well understood. It has been known for a while that $\Sigma$ has finite Morse index if and only if $\Sigma$ has finite total curvature (\cite{Fischer-Colbrie}, \cite{GulliverLawson86}, \cite{Gulliver86}, \cite{Ros06}, \cite{Ross92}). More recently, O. Chodosh and the first author \cite{ChodoshMaximo16,ChodoshMaximo},  generalized these results by proving an explicit lower bound on the Morse index that depends on the genus
and number of ends, counting multiplicity. The case of branched immersions was later considered by Karpukhin (\cite[Proposition 2.3]{Karpukhin19}) and Meeks-Perez (\cite[Equation (3.1)]{MeeksPerez22}).

An important feature of the above-cited works is that the Jacobi operator of $\Sigma$, which comes from the second variation of the area, can be reduced to $-\Delta_\Sigma + 2K_\Sigma$, which is intrisic to the surface. The relationship between the Morse index and total curvature of $\Sigma$ can thus be motivated by long-standing questions about the eigenvalues of Schrodinger operators as in \cite{GNY,Grigor'yanYau}.

The starting point of our work concerns similar questions for surfaces in hyperbolic 3-space $\mathbb H^3$. As a result from F. Martin and B. White \cite{MaWh}, any open, orientable, connected surface can be properly embedded in $\mathbb H^3$ as a complete area-minimizing surface. Thus the concordance between the theory of minimal surfaces in $\mathbb H^3$ and in $\mathbb R^3$
is not a very good one. 

On the other hand, the results about minimal surfaces mentioned above use at many levels that minimal surfaces in $\mathbb{R}^3$ can be described through holomorphic functions and differential forms defined in $\Sigma$, which is known in the literature as \emph{Weiestrass representation}. In particular, one can verify that $\Sigma$ is minimal if and only if its Gauss map is weakly (anti) conformal.  For surfaces in hyperbolic space $\mathbb{H}^3$, Bryant showed in his seminal paper \cite{Bryant87} that the good analog of minimal surfaces for $\mathbb{H}^3$ are the solutions to a constrained variational problem for the area, the constant mean curvature one surfaces (CMC-1). For instance, a surface in $\mathbb{H}^3$ is CMC-1 if and only if its (hyperbolic) Gauss map is weakly (anti) conformal. Bryant also described the analog of the Weierstrass representation for these surfaces, which are known in the literature as \emph{Bryant surfaces}. The description of the Bryant-Weiestrass representation has been further explored by other authors (\cite{Small94}, \cite{KokubuUmeharaYamada}, \cite{RossmanUmeharaYamada}). Turns out that the analogy between Euclidean minimal surfaces and hyperbolic CMC-1 surfaces follows a $1$-to-$1$ correspondence known in the literature as \emph{Lawson correspondence} \cite{Lawson70}, which describes a correspondence between intrinsic CMC surfaces in space forms. Indeed, this correspondence matches two surfaces if they have the same induced metric and the same traceless second fundamental form, which can be understood intrinsically as two tensors in a surface satisfying the Gauss-Codazzi equations.

While we will use Lawson's correspondence to establish our results in hyperbolic geometry, it should be noted that extrinsic properties are not necessarily preserved, as it is the case for immersions or embeddings. Indeed, any immersed minimal surface in $\mathbb{R}^3$ has total curvature equal to an integer multiple of $4\pi$, which is false for immersed Bryant surfaces (\cite[Example 2]{Bryant87}). In the context of immersed Bryant surfaces, Lima and Rossman \cite{LimaRossman} have studied their index, while Rossman, Umehara and Yamada \cite{RossmanUmeharaYamadaIrreducible} have studied families of Euclidean minimal surfaces that deform into hyperbolic space as (immersed) Bryant surfaces. In contrast, we will expand the family of surfaces under study by considering \emph{equivariant} surfaces.

We say that $\Sigma$ is an equivariant (or intrinsic) CMC surface in a space form $M^3(\kappa)$ if we have a CMC immersion of its universal cover $\widetilde{\Sigma}$ in $M^3(\kappa)$ that is equivariant by an action of $\pi_1(\Sigma)$ on $M^3(\kappa)$ by isometries. Observe that under these assumptions the first and second fundamental forms are well-defined on $\Sigma$, and as we will see, this is the correct setup to establish Lawson's correspondence. These surfaces generalize immersed minimal surface (where $\pi_1(\Sigma)$ acts trivially) and they arise as critical points for the area functional when considering equivariant surfaces (not necessarily CMC) under equivariant perturbations that preserve the action of $\pi_1(\Sigma)$ on $M^3(\kappa)$. While in the literature the term Bryant surface has been used almost exclusively for either embedded or immmersed surfaces in $\mathbb{H}^3$, we will keep this notation for equivariant surfaces, while adding reminders whenever needed.

In order to discuss the second variation of area, we have to make a distinction between two-sided and one-sided equivariant surfaces. Fix along $\widetilde{\Sigma}$ one of the two continuous choices for normal vectors and denote it by $\overrightarrow{N}$ (the following does not depend on the choice between $\overrightarrow{N}$ and $-\overrightarrow{N}$). We say that $\Sigma$ is \emph{two-sided} if the action of $\pi_1(\Sigma)$ sends $\overrightarrow{N}$ to itself, otherwise we say that $\Sigma$ is {\it one-sided}. Observe that if $\Sigma$ is one-sided then it has a two-sided double cover $\widehat{\Sigma}$, namely the normal cover given by the subgroup of $\pi_1(\Sigma)$ that fixes $\overrightarrow{N}$. If $\Sigma$ is two-sided then for any compactly supported function $u$ in $\Sigma$ we consider the second variation of area given by the vector field $\widetilde{u}\overrightarrow{N}$, where $\widetilde{u}$ denotes the lift of $u$ to $\widetilde{\Sigma}$. If $\Sigma$ is one-sided and $\tau$ is the involution defined in the double cover $\widehat{\Sigma}$, then we consider compact supported functions $u$ in $\widehat{\Sigma}$ with $u\circ\tau = -u$, and we take the second variation of area along $\widetilde{u}\overrightarrow{N}$. Observe that as isometries of space form are determined by the image of an orientable orthonormal basis, then if $\Sigma$ is orientable then it has to be two-sided.

In the case of equivariant minimal surfaces in $\mathbb{R}^3$ or equivariant Bryant surfaces in $\mathbb{H}^3$, the Jacobi operator is still given by $-\Delta_\Sigma + 2K_\Sigma$, although applied to different spaces of functions. Indeed, the Morse index for a minimal surface is calculated while considering all functions $f$ in $\Sigma$ with compact support, while the Morse index for a CMC surface (with non-zero mean curvature) is calculated for compactly supported functions satisfying $\int_\Sigma f = 0$. This is because for minimal surfaces we consider all perturbations of area, while for CMC surfaces we consider all perturbations of area with the same enclosed volume. We will denote the unrestricted index by ${\rm Ind}(\Sigma)$ and the restricted index by ${\rm Ind}_{res}(\Sigma)$. As by work of Barbosa and B\'erard \cite{BarbosaBerard} and Lima and Rosmann \cite{LimaRossman} we have that ${\rm Ind}_{res}(\Sigma)={\rm Ind}(\Sigma)$ if $\Sigma$ is complete non-compact and ${\rm Ind}(\Sigma)-1\leq {\rm Ind}_{res}(\Sigma) \leq {\rm Ind}(\Sigma)$ if $\Sigma$ is closed, from now on we will focus our attention to the study of the unrestricted index ${\rm Ind}(\Sigma)$.

This article has two main objectives. On one hand, we generalize the lower bounds for Morse index of \cite{ChodoshMaximo} and \cite{Karpukhin19} for a certain subfamily of equivariant surfaces we will call \emph{framed} surfaces. A framed Euclidean minimal surface $\Sigma$ can be described in the following equivalent ways. We say that an equivariant minimal surface $\Sigma$ in $\mathbb{R}^3$ is $\emph{framed}$ if it satisfies any of the following 
\begin{enumerate}
    \item $\pi_1(\Sigma)$ acts by translations in $\mathbb{R}^3$.
    \item The Gauss map is well-defined in $\Sigma$.
    \item The fundamental harmonic forms $dx_1, dx_2, dx_3$ are well defined in $\Sigma$.
\end{enumerate}
The term framed was coined because of the last condition. We will prove:

\begin{maintheorem}\label{mainthm:index}
Let $\Sigma$ be a framed, two-sided and possibly branched minimal surface, and let $D=\sum_P n_P.P$ be the divisor defined as $n_P=-m$ if $P$ is an end of order $m$, and $n_P=m$ if $P$ is a branching point of order $m$ ($n_P=0$ otherwise). Let $h^1(D)$ denote the complex dimension of the space of holomorphic 1-forms with divisor $\geq D$. Then
\[{\rm Ind}(\Sigma) \geq \frac{2h^1(D)-3}{3}.
\]
\end{maintheorem}

As pointed out by Meeks-Perez (\cite[Equation (3.1)]{MeeksPerez22}) we can apply Riemann-Roch to the above inequality to bound from below the Morse index by an explicit linear function of the Euler characteristic of $\Sigma$ and the order of the ends and branched points. Namely, we obtain
\[{\rm Ind}(\Sigma) \geq \frac{2(g(\Sigma)-1)-2{\rm deg}(D)-3}{3}.
\]
where $g(\Sigma)$ denotes the genus and ${\rm deg}(D) = \sum_P n_P$ denotes the degree of the divisor.

Our next objective is to state an analogous lower bound for the Morse index of framed Bryant surfaces, which was the original motivation for this project. It turns out that Lawson's correspondence translates very well all items included in the assumptions, proof and conclusion of Theorem \ref{mainthm:index}. More explicity, we observe the following.

\begin{maintheorem}\label{thm:Lawsonmirror}
    The Lawson's correspondence of an equivariant Euclidean minimal surface and an equivariant Bryant surface also matches, in addition of the induced metric and the traceless second fundamental form, the following:
    \begin{itemize}
        \item The Jacobi operator $-\Delta_\Sigma + 2K_\Sigma$.
        \item The (unrestricted) Morse index ${\rm Ind}(\Sigma)$.
        \item The Euclidean Gauss map with the secondary hyperbolic Gauss map.
        \item The property of being two-sided or one-sided. 
         \item The canonical Euclidean harmonic forms $dx_1, dx_2, dx_3$ with the canonical harmonic forms $w^1, w^2, w^3$ coming from the space of orthonormal frames in $\mathbb{H}^3$(see Equation \eqref{eq:harmonicformsPSL2C}).
        \item The property of being framed or not.
        \item The order of any end or branched point.
    \end{itemize}
\end{maintheorem}

Thus we have the exact same result to Theorem \ref{mainthm:index} for framed Bryant surfaces if we replace minimal by Bryant at every instance. Moreover, as we will see later, the proof of Theorem \ref{mainthm:index} is written entirely intrinsically, so we are actually proving both results at once. Hence in the relevant sections we will simply refer to framed surfaces, which will denote either Euclidean minimal or Bryant surfaces.

For the case ${\rm Ind}(\Sigma)=0$ Theorem \ref{mainthm:index} implies that $\Sigma$ is either a plane, an annuli or a tori. Hence it is easy to conclude the following corollary.

\begin{maincor}
    The only framed, two-sided, minimal surfaces without branch points that are stable are the Euclidean plane, Euclidean cylinder and Euclidean tori. Equivalently, the only framed, two-sided, no branched Bryant surfaces are the horosphere, horocylinder and horotori.
\end{maincor}

\subsection{Some examples} For ${\rm Ind}(\Sigma)=1$ we have that Theorem \ref{mainthm:index} is sharp, as the Schwarz's P surface is Euclidean minimal, framed, has genus equal to $3$ and index equal to $1$. Another example is given by Scherk's doubly periodic surface, whose quotient has genus $0$, $4$ ends and index $1$ (see \cite[Corollary 15]{MontielRos}). By Remark \ref{remark:rotation} and the description of framedness, each of these surfaces defines a $1$-parameter family of examples with ${\rm Ind}(\Sigma)=1$ where our inequality is sharp. 

Another family of examples are the \emph{catenoid cousins}, whose index was studied in \cite{LimaRossman}. This is a family of Bryant surfaces on annuli depending on a real parameter $\mu>-\frac12, \mu\neq 0$, and as we will see in \ref{subsec:examples} they are framed if and only if $2\mu$ is an integer. By \cite[Theorem 6.1]{LimaRossman} we have that for $\mu<0$ the catenoid cousins have index $1$ which is the lower bound given by Theorem \ref{mainthm:index}. Then for the values $-\frac12<\mu<0$ we have embedded Bryant surfaces that are not framed but have low index, which raises the question of whether a similar bound as in Theorem \ref{mainthm:index} can be shown for embedded or immersed Bryant surfaces. 

Another interesting question would be the classification of framed Bryant surfaces of low index, even under some symmetry assumption. For instance in the index 1 case, Theorem \ref{mainthm:index} restricts the topology for non-branched surfaces to $g+e\leq 4$, where $g$ is the genus and $e$ is the sum of orders of all ends. One may wonder if some ideas in the vein of \cite{Chen} could lead to a classification under a suitable symmetry assumption.

\subsection{The one-side case} Finally, in the same vein, we will also prove the following theorem for one-sided framed surfaces.

\begin{maintheorem}\label{mainthm:non-orientable}
Let $\Sigma$ be a framed, one-sided and possibly branched minimal (or Bryant) surface, $\widehat{\Sigma}$ the two-sided cover of $\Sigma$ and let $D=\sum_P n_P.P$ be the divisor in $\widehat{\Sigma}$ defined as $n_P=-m$ if $P$ is an end of order $m$, and $n_P=m$ if $P$ is a branching point of order $m$ ($n_P=0$ otherwise). Let $h^1(D)$ denote the complex dimension of the space of holomorphic 1-forms with divisor $\geq D$. Then
\[{\rm Ind}(\Sigma) \geq \frac{h^1(D)-3}{3}.
\]
\end{maintheorem}

The article is organized as follows. Section \ref{sec:Background} covers the relevant background. Subsection \ref{subsec:Bryant surfaces} covers the equivalent definitions for intrinsic Bryant surfaces, their correspondence to Euclidean minimal surfaces, the description of ends and framing, and some basic examples (Subsection \ref{subsec:examples}). In particular Subsection \ref{subsec:Bryant surfaces} proves all items of Theorem \ref{thm:Lawsonmirror}.  Subsection \ref{subsec:Schwarzian derivative} covers briefly the Schwarzian derivative, which is key to describe the Gauss map of Bryant surfaces, taking special care in the description around branched points. Section \ref{sec:Weighted space} introduces a weighted inner product for functions in a framed surface, which will help us to associate eigenfunctions of a weighted stability operator to the Morse index. Finally, Section \ref{sec:ProofA} deals with the proof of Theorem \ref{mainthm:index} while Section \ref{sec:ProofB} addresses Theorem \ref{mainthm:non-orientable}. These last three sections follow closely the proofs of \cite[Theorem 1.1, Theorem 1.3]{ChodoshMaximo} (and by extension work of Fischer-Colbrie \cite{Fischer-Colbrie} and Ros \cite{Ros06}), where the main difference is that we take an intrinsic approach.

\section*{Acknowledgements} D.M. would like to thank O. Chodosh for countless conversations about the index of minimal surfaces. F.V.P. is thankful to L. Ambrozio, O. Chodosh and C. Viana for their interest and helpful comments.

\section{Background}\label{sec:Background}

In the following Subsection, we will introduce four equivalent definitions for a Bryant surface (Definitions \ref{defi:intrinsic}, \ref{defi:cmc1}, \ref{defi:Bryant}, \ref{defi:Small}) and describe its basic properties in the context of each definition.

\subsection{Bryant surfaces}\label{subsec:Bryant surfaces}
Let $\Sigma$ be an immersed surface in a $3$-dimensional space form. Let us assume for simplicity that $\Sigma$ is orientable (and in consequence two-sided), passing through a double cover if needed. If one expresses the second fundamental form in isothermal coordinates as $(h_{ij})$ then the Codazzi equation is equivalent to
\begin{equation}\label{eq:codazzi}
\begin{split}
\nabla_1 h_{21} &= \nabla_2 h_{11} \\
\nabla_1 h_{22} &= \nabla_2 h_{12}.
\end{split}
\end{equation}
Observe that Equation (\ref{eq:codazzi}) is satisfied by the metric $(g_{ij})$ thanks to the compatibility of the Levi-Civita connection. Then if $\Sigma$ has constant mean curvature and $(h^0_{ij})$ denotes the traceless second fundamental form, then Equation (\ref{eq:codazzi}) is equivalent to
\begin{equation}\label{eq:CauchyRiemann}
\begin{split}
\nabla_1 h^0_{21} &= \nabla_2 h^0_{11} \\
-\nabla_1 h^0_{11} &= \nabla_2 h^0_{12},
\end{split}
\end{equation}
which are the Cauchy-Riemann equations for $h^0_{11} - ih^0_{12}$ in conformal coordinates. Hence the Codazzi equation is equivalent to the quadratic form $\sigma = (h^0_{11}-ih^0_{12})\dd z^2$ being holomorphic.

Now, if we express $(g_{ij})$ as $e^{2u}g_{0}$, where $g_{0}$ is the uniformized metric (i.e. the complete constant curvature metric conformal to $(g_{ij})$ with curvature $\pm1, 0$), and $e^{2u}$ is the conformal factor, then the Gauss equation is equivalent to 
\begin{equation}\label{eq:Gauss}
    e^{-2u}(-\Delta_{g_{0}}u-K_0) = \kappa +  (c^2 -|\sigma|^2_{g_{0}}e^{-4u}) = (c^2+\kappa) -|\sigma|^2_{g_{0}}e^{-4u},
\end{equation}
where $K_0$ is the curvature of $g_0$, $-\Delta_{g_{0}}$,  $|\cdot|_{g_{0}}$ are the laplacian and norm induced by $g_{0}$, $\kappa$ is the constant curvature of the ambient space form and $c$ is the (constant) mean curvature of $\Sigma$.

Conversely, if $\Sigma$ is a simply connected surface with uniformized metric $g_{0}$, a holomorphic quadratic differential $\sigma$ and a conformal factor $e^{2u}$ satisfying (\ref{eq:Gauss}), then by a theorem of Cartan there exists a constant mean curvature $c$ immersion $\Sigma\hookrightarrow M^3(\kappa)$ where the metric is given by $e^{2u}g_{0}$ and the second fundamental form is given by ${\rm Re}(\sigma) + cg_{0}$. Moreover, such immersion is unique up to composition with an isometry of $M^3(\kappa)$.

In particular, if $\Sigma$ is any surface (not necessarily simply connected) and $(\Sigma,g_{0},\sigma,e^{2u})$ solves (\ref{eq:Gauss}), then $(\widetilde{\Sigma},\widetilde{g_{0}},\widetilde{\sigma},e^{2\widetilde{u}})$ also solves (\ref{eq:Gauss}). Hence this last expression is realized by an immersion $e_0:\widetilde{\Sigma}\hookrightarrow M^3(\kappa)$ (up to an isometry of $M^3(\kappa)$). But given the invariance of $(\widetilde{g_{0}},\widetilde{\sigma},e^{2\widetilde{u}})$ by any deck transformation $\phi\in \pi_1(\Sigma)$, we have that there exists $\rho(\phi)\in {\rm Isom}^+M^3(\kappa)$ so that $e_0(\phi z) = \rho(\phi)e_0(z)$. This constructs a representation $\rho:\pi_1(\Sigma)\rightarrow {\rm Isom}^+M^3(\kappa)$ so that $e_0$ is $\rho$-equivariant. Conversely, any constant mean curvature $c$ immersion $\widetilde{\Sigma}\rightarrow$ that is $\rho$-equivariant for some representation $\rho:\pi_1(\Sigma)\rightarrow {\rm Isom}^+M^3(\kappa)$ gives a solution to (\ref{eq:Gauss}) on $\Sigma$.

Observe also that for any given real constant $c_0$ and parameter $\kappa<c_0$, we have a $1$-to-$1$ correspondence among $\sqrt{c_0-\kappa}$ constant mean curvature surfaces in $M^3(\kappa)$. Moreover given $(\Sigma,g_{0}, \sigma, e^{2u})$ satisfying $e^{-2u}(-\Delta_{g_{0}}u-K_0) = c_0 -|\sigma|^2_{g_{0}}e^{-4u}$ we can select isometric immersions $\widetilde{\Sigma}\hookrightarrow M^3(\kappa)$ that vary smoothly on $\kappa$. Such correspondence is known in the literature as \emph{Lawson's correspondence} \cite{Lawson70}.

Hence the following two definition are equivalent.

\begin{defi}\label{defi:intrinsic}
Let $\Sigma$ be a surface with uniformized metric $g_{0}$. We say that a function $u$ and a holomorphic quadratic differential $\sigma$ define a (intrinsic) Bryant surface if they solve the Gauss equation in $\HH$
\begin{equation}\label{eq:Gausshyp}
    e^{-2u}(-\Delta_{g_{0}}u-K_0) = -1 + (1+|\sigma|_{g_{0}}e^{-2u})(1-|\sigma|_{g_{0}}e^{-2u}) = -|\sigma|^2_{g_{0}}e^{-4u} 
\end{equation}
\end{defi}

\begin{defi}\label{defi:cmc1}
Let $\Sigma$ be a surface. We say that $\Sigma$ is a Bryant surface given that we have
\begin{enumerate}
    \item A homomorphism $\rho:\pi_1(\Sigma)\rightarrow \PSL \simeq {\rm Isom}^+\HH$.
    \item A constant mean curvature $1$ immersion $e_0:\tilde{\Sigma}\rightarrow\HH$ that is $\rho$-equivariant. More precisely, for any $\phi\in\pi_1(\Sigma)$ we have $e_0(\phi.z)=\rho(\phi).e_0(z)$.
\end{enumerate}
\end{defi}
\begin{remark}
    Observe that in Definition \ref{defi:intrinsic} $K_0$ cannot be equal to $1$. Indeed, if such where the case then $\Sigma$ would be a immersed sphere in $\mathbb{H}^3$. By taking a horosphere tangent to $\Sigma$ so that $\Sigma$ lies in its interior, we observe that this case cannot occur, as horosphere have constant mean curvature $1$.
\end{remark}
\begin{remark}
Observe that in Definition \ref{defi:cmc1} the homomorphism $\rho$ and the immersion $e_0$ are defined up to conjugation and left multiplication by an element in $\PSL \simeq {\rm Isom}^+\HH$, respectively. To keep exposition simple, we will omit this technical point and assume it as understood in future discussions.
\end{remark}

\begin{remark}\label{remark:rotation}
    Observe if $(u,\sigma)$ satisfies Definition \ref{defi:intrinsic} then for a real parameter $\theta$ the one parameter family $(u,e^{i\theta}\sigma)$ is a family of Bryant surfaces with the same intrinsic metric, and consequently the same values for principal curvatures. The change happens at the level of the principal directions foliations, which rotate by $\theta$.
\end{remark}

By the previous discussion, Bryant surfaces are in $1$-to-$1$ correspondence with intrinsic minimal surfaces in $\RR^3$, i.e. $\pi_1(\Sigma)$ equivariant minimal immersions of $\widetilde{\Sigma}$ into $\mathbb{R}^3$ acting by isometries.

On his seminal paper, Bryant \cite{Bryant87} characterizes the analogous to the Weierstrass representation of Definition \ref{defi:cmc1}. It is well-known that $\PSL \simeq {\rm Isom}^+\HH$ can be identified with the oriented frame bundle of $\HH$ and the basepoint projection map $\pi:\PSL\rightarrow\HH$ can be identified with $\pi(A) = A {}^{t}\overline{A}$. We say that a holomorphic map $F:\widetilde{\Sigma}\rightarrow\PSL$ is \emph{null} if $\det(F^{-1}\dd F)\equiv 0$. Then Bryant shows (\cite[Theorem A]{Bryant87}) that a simply connected surface $e_0:\widetilde{\Sigma}\hookrightarrow\HH$ has constant mean curvature $1$ if and only if there exists a null-immersion $F:\widetilde{\Sigma}\rightarrow \PSL$ satisfying $\pi\circ F = e_0$, and such map is unique up to right multiplication by an element of $\SU <\PSL$. Hence we have our third equivalent definition for a Bryant surface.

\begin{defi}\label{defi:Bryant}
Let $\Sigma$ be a surface. We say that $\Sigma$ is a Bryant surface given that we have
\begin{enumerate}
    \item Homomorphisms $\rho:\pi_1(\Sigma)\rightarrow \PSL$, $h:\pi_1(\Sigma)\rightarrow \SU$.
    \item A null-immersion $F:\widetilde{\Sigma}\rightarrow\PSL$ that is $(\rho,h)$-equivariant. More precisely, for any $\phi\in\pi_1(\Sigma)$ we have $F(\phi.z)=\rho(\phi).F(z).h(\phi)$.
\end{enumerate}
\end{defi}

To see that Definition \ref{defi:Bryant} follows from Definition \ref{defi:cmc1}, we have that if $F:\widetilde{\Sigma}\rightarrow\PSL$ is a null-immersion lift of the $\rho$-equivariant immersion $e_0:\widetilde{\Sigma}\hookrightarrow\HH$ and $\phi\in\pi_1(\Sigma)$ is a deck transformation, then $\rho(\phi)^{-1}F(\phi.z)$ is also a null-immersion lift of $e_0$. Hence there exists $h(\phi)\in\SU$ so that $\rho(\phi)^{-1}F(\phi.z) = F(z)h(\phi)$, and this defines a representation $h:\pi_1(\Sigma)\rightarrow\SU$ so that $F$ is $(\rho,h)$-equivariant.

Conversely, given a Bryant surface in the sense of Definition \ref{defi:Bryant} we have that $F{}^{t}\overline{F}(\phi.z) = \rho(\phi)F{}^{t}\overline{F}(z)\rho(\phi)^{-1}$, so then the immersion $e_0=\pi\circ F$ is a $\rho$-equivariant CMC-1 immersion $\widetilde{\Sigma}\hookrightarrow\HH$.

For a null-immersion $F:\widetilde{\Sigma}\rightarrow\PSL$, \cite[Theorem A]{Bryant87} says that we can write the $\PSL$ valued form $F^{-1}dF$ as
\begin{equation}
    F^{-1}dF = \begin{pmatrix} pq & -q^2\\  p^2 & -pq \end{pmatrix}\varphi,
\end{equation}
where $p,q$ are smooth function on $\widetilde{\Sigma}$ satisfying $p\overline{p}+q\overline{q}=1$ with $q/p$ meromorphic, and $\varphi$ is a $(1,0)$-form on $\widetilde{\Sigma}$. Given this notation one can also compute

\begin{enumerate}
    \item The Gauss map as $dF_1/dF_3$, by identifying $\partial_\infty\HH$ with $\overline{\CC}$.
    \item The induced metric $ds^2$ by $e_0$ from $\HH$ as $ds^2 = \varphi\overline{\varphi}$
    \item The holomorphic quadratic differential $\sigma$ defined by the traceles second fundamental form as $\sigma = 2(pdq-qdp)\varphi$
\end{enumerate}

From here we can see as well how Definition \ref{defi:intrinsic} follows from Definition \ref{defi:Bryant}, as $ds^2=e^{2u}g_{0}$ and $\sigma$ solve \ref{eq:Gausshyp}. In particular $F$ is required to be an immersion so that $ds^2$ is non-degenerate. Hence we will see that $\Sigma$ is a \emph{branched} Bryant surface if we have $(\rho,h)$-equivariant non-constant null holomorphic map $F:\widetilde{\Sigma}\rightarrow\PSL$. Hence $\widetilde{\Sigma}$ has a branched isometric immersion of constant mean curvature $1$ into $\HH$, where the branch points are isolated in $\Sigma$. Moreover, as $\varphi$ is a $(1,0)$ form and $ds^2=\varphi\overline{\varphi}$, then each branched point has an integer multiple of $2\pi$ as cone angle. Definitions \ref{defi:intrinsic} and \ref{defi:cmc1} are also revised accordingly for the branched case.

The next definition follows work of Small \cite{Small94} that explores further the Weierstrass representation for Bryant surfaces.

\begin{defi}\label{defi:Small}
Let $\Sigma$ be a surface. We say that $\Sigma$ is a Bryant surface given that we have
\begin{enumerate}
    \item Homomorphisms $\rho:\pi_1(\Sigma)\rightarrow \PSL$, $h:\pi_1(\Sigma)\rightarrow \SU$.
    \item A pair of non-constant meromorphic functions $f,g:\widetilde{\Sigma}\rightarrow\overline{\CC}$ so that $f$ is $\rho$-equivariant and $g$ is $h^{-1}$-equivariant. More precisely, for any $\phi\in\pi_1(\Sigma)$ we have $f(\phi.z)=\rho(\phi).f(z)$, $g(\phi.z)=h^{-1}(\phi).g(z)$
    \item\label{item:criticalpoints} $f$ and $g$ have the same critical points counting multiplicity, and at such points $\frac{df}{dg}$ (or changing $f, g$ by $1/f, 1/g$ depending on which one has a pole) also has a critical point of at least the common multiplicity of $f$ and $g$.
\end{enumerate}
or that $\Sigma$ is locally modeled on an horosphere, in which case the maps $f$ and $g$ are constant.
\end{defi}

\begin{remark}
    Given the conditions given in Definition \ref{defi:Small}, one can say coloquially (excluding horospheres) that Bryant surfaces correspond to conformally compatible projective and spherical structures on $\Sigma$, where $f$ and $g$ are the developing maps of each structure. We will see that $f$ will correspond to the hyperbolic Gauss map of the Bryant surface, while $g$ corresponds to what is known in the literature as the \it{secondary Gauss map} (see for instance \cite{RossmanUmeharaYamada}).
\end{remark}

To see the equivalence of Definition \ref{defi:Small} with the previous definitions we have to work a bit more. Let then $f,g:\widetilde{\Sigma}\rightarrow\overline{\CC}$ be non-constant complex analytic functions. In concordance with \cite[Section 3]{Small94} we define $\omega:\widetilde{\Sigma}\rightarrow\PSL$ as $\omega = \big(\begin{smallmatrix}
  \alpha & \beta\\
  \gamma & \delta
\end{smallmatrix}\big)$ where
\begin{equation}\label{eq:omegadefinition}
\begin{split}
    \alpha &= \left( \frac{df}{dg}\right)^{1/2} -\frac{f}2  \left(\frac{df}{dg} \right)^{-3/2} \left( \frac{d^2f}{dg^2}\right) \\
    \beta &= f\bigg\lbrace \left( \frac{df}{dg}\right)^{-1/2} + \frac{g}2\left(\frac{df}{dg} \right)^{-3/2} \left( \frac{d^2f}{dg^2}\right) \bigg\rbrace - g\left( \frac{df}{dg}\right)^{1/2}\\
    \gamma &= -\frac12 \left(\frac{df}{dg} \right)^{-3/2} \left( \frac{d^2f}{dg^2}\right)\\
    \delta &= \left( \frac{df}{dg}\right)^{-1/2} + \frac{g}2\left(\frac{df}{dg} \right)^{-3/2} \left( \frac{d^2f}{dg^2}\right)
\end{split}
\end{equation}

Denote by $C(f)$, $C(g)$ critical points (with multiplicity) of $f$ and $g$. Then one observes that while $\omega$ includes square root terms in its coordinates and it is defined away of $C(f), C(g)$, it is well-defined in all of $\Sigma\setminus(C(f) \cup C(g))$ since the image of $\omega$ is in $\PSL$. Moreover, $\omega$ extends to a point $z_0\in C(f)\cup C(g)$ if and only if satisfy condition (\ref{item:criticalpoints}) in Definition \ref{defi:Small} (see \cite[Section 3]{Small94}). Hence if $f$, $g$ have removable singularities in an punctured end of $\Sigma$ (i.e. a neighbourhood conformally equivalent to a punctured disk) and do not satisfy (\ref{item:criticalpoints}), then $\omega$ will have a singularity at that end. We refer to this type of end as \emph{regular} (contrast for instance with \cite{UmeharaYamada93}, where only $f$ is required to extend to the end).

We also define $S\lbrace f,g\rbrace$,  the \emph{Schwazian derivative} of $f$ with respect to $g$, as the $(2,0)$ meromorphic form in $\widetilde{\Sigma}$
\begin{equation}\label{eq:Schwarziandefinition}
    S\lbrace f,g\rbrace = \Bigg\lbrace\left(\frac{df}{dg} \right)^{-1} \left( \frac{d^3f}{dg^3}\right) - \frac32 \left(\frac{df}{dg} \right)^{-2} \left( \frac{d^2f}{dg^2}\right)\Bigg\rbrace dg^2.
\end{equation}
The Schwarzian derivative satisfy the following properties
\begin{enumerate}
    \item $S\lbrace g,f\rbrace = -S\lbrace f,g\rbrace$
    \item For any $\theta\in\PSL$ we have that $S\lbrace \theta\circ f,g\rbrace = S\lbrace f,\theta\circ g\rbrace = S\lbrace f,g\rbrace$.
    \item $S\lbrace f,g\rbrace\equiv 0$ if and only if there exists $\theta\in\PSL$ so that $f=\theta\circ g$.
\end{enumerate}
In particular, if $f,g:\widetilde{\Sigma}\rightarrow\overline{\CC}$ are equivariant by maps in $\PSL$ then $S\lbrace f,g\rbrace$ is well-defined as a $(2,0)$ meromorphic form in $\Sigma$. And since $f$ and $g$ have the same critical points counting multiplicity, then $S\lbrace f,g\rbrace$ is in fact holomorphic in $\widetilde{\Sigma}$. Moreover, by the discussion of regular ends we have that at those points $S\lbrace f,g\rbrace$ has zeros of at least the same order as the multiplicity of the critical points of $g$.

It is a straightforward calculation to find that
\begin{equation}
    d\omega = -\frac12 \begin{pmatrix} f & -fg\\  1 & -g \end{pmatrix} \left( \frac{df}{dg}\right)^{-1/2}S\lbrace f,g\rbrace dg^{-1}
\end{equation}
From this formula and the properties of the Schwarzian derivative, we see that $\omega$ is constant if and only if there exists $\theta\in\PSL$ so that $f=\theta\circ g$.

We can also easily find
\begin{equation}
    \omega^{-1}d\omega = -\frac12 \begin{pmatrix} g & -g^2\\  1 & -g \end{pmatrix} S\lbrace f,g\rbrace dg^{-1},
\end{equation}
from where we can easily see that $\omega:\widetilde{\Sigma}\rightarrow\PSL$ is a null holomorphic map. For the map to $\mathbb{H}^3$ defined by the projection of $\omega$ we can compute

\begin{enumerate}
    \item The Gauss map as $f$
    \item The induced metric $ds^2$ as $ds^2 = \frac14(1+|g|^2)^2|S\lbrace f,g \rbrace|^2|dg|^{-2}$
    \item The holomorphic quadratic differential $\sigma$ as $\sigma = -S\lbrace f,g\rbrace$
\end{enumerate}
With all these details explained, we are ready to prove that Definition \ref{defi:Small} is equivalent to Definitions \ref{defi:intrinsic}, \ref{defi:cmc1} and \ref{defi:Bryant}.

Given Definition \ref{defi:Small} we have that since $f,g:\widetilde{\Sigma}\rightarrow\overline{\CC}$ are $\rho$ and $h^{-1}$ equivariant respectively, then $\sigma = -S\lbrace f,g\rbrace$ is a well-defined quadratic holomorphic differential in $\Sigma$. Moreover, since $h$ has image in $\SU$ then $g^*(g_{\mathbb{S}^2}) = \frac{4|dg|^2}{(1+|g|^2)^2}$ is a well-defined branched metric in $\Sigma$, and because the Schwarzian vanishes at least at the same order as the multiplicity of the critical points of $g$ then the metric $ds^2 = \frac14(1+|g|^2)^2|S\lbrace f,g \rbrace|^2|dg|^{-2}$ is also a well-defined branched metric in $\Sigma$. Since $\omega$ is a null holomorphic map then $ds^2=e^{2u}g_{0}$ and $\sigma$ define a Bryant surface in the sense of Definition \ref{defi:intrinsic}. By the equivalence with Definition \ref{defi:Bryant}, it follows that there exists homomorphisms $\rho_1:\pi_1(\Sigma)\rightarrow \PSL$, $h_1:\pi_1(\Sigma)\rightarrow \SU$ so that $\omega$ is $(\rho_1,h_1)$ equivariant. As the Gauss map of $\omega$ is given by $f$, then it follows that $f$ is both $\rho$ and $\rho_1$ equivariant, and subsequently $\rho=\rho_1$. Similarly, the $\PSL$ valued form $\omega^{-1}d\omega$ is both $h$ and $h_1$ equivariant, so it follows that $h=h_1$.

Conversely, let us start from Definition \ref{defi:Bryant}. We define $g$ as the meromorphic function $q/p$. We will assume that $g=q/p$ is not constant, as otherwise we will have that $\sigma=2(pdq-qdp)\varphi\equiv0$ and consequently that $\Sigma$ is umbilic, hence it is locally isometric to an horosphere. Since $F^{-1}dF$ is $h$-equivariant, then it follows that $g$ is $h$-equivariant. Then we define $f:\widetilde{\Sigma}\rightarrow\overline{\CC}$ as the solution of $S\lbrace f,g\rbrace = -\sigma$. This solution will satisfy condition (\ref{item:criticalpoints}). Since $\sigma$ is defined in $\Sigma$, it follows that $f$ is $\rho_1$ equivariant for some homomorphism $\rho_1:\pi_1(\Sigma)\rightarrow\PSL$. Hence the null-immersion $\omega=\omega(f,g)$ satisfies $\omega^{-1}d\omega=F^{-1}dF$. As then $\omega$ and $F$ define the same intrinsic Bryant surface of Definition \ref{defi:intrinsic}, it follows that there exists $\theta\in\PSL$ so that $\omega=\theta F$. As the Gauss map of $\omega$ is $\rho_1$ equivariant and the Gauss map of $F$ is $\rho$-equivariant, it follows that $\rho=\rho_1$.

Observe that for a branched Bryant surface in the sense of Definition \ref{defi:Small} branched points are characterized by zeros of $\sigma=-S\lbrace f,g\rbrace$ of higher multiplicity than the zeros of $dg$. Hence a surface in the sense of Definition \ref{defi:Small} is immersed if and only if $\sigma$ and $dg$ have the same zeros counting multiplicities.

As we have not only proven equivalence of definitions but concordance on notation, from now on we will use interchangeably either definition and we will use common notation while referring to the metric $ds^2$, holomorphic quadratic differential $\sigma$ and homomorphisms $\rho:\pi_1(\Sigma)\rightarrow \PSL$, $h:\pi_1(\Sigma)\rightarrow \SU$.

In general, given a meromorphic map $g:\widetilde{\Sigma}\rightarrow\overline{\CC}$ equivariant by some representation of $\pi_1(\Sigma)$ into $\SU$ and a holomorphic quadratic differential $\sigma$ in $\Sigma$, the solution of $S\lbrace f,g\rbrace = -\sigma$ is given by a meromorphic $\PSL$ equivariant function $f:\widetilde{\Sigma}\rightarrow\overline{\CC}$ satisfying that it has the same critical points as $g$, counting multiplicities. If at any of those points $\frac{df}{dg}$ does not have the same multiplicity as $g$ then the null-immersion $\omega(f,g)$ has a pole singularity at such point. Hence this construction gives an equivariant Bryant surface with regular ends. Hence we will say that a regular end of a Bryant surface is \emph{type-1} if $S\lbrace f,g\rbrace$ extends holomorphically

More generally, we would like to describe which meromorphic quadratic differential $\sigma$ can appear as $S\lbrace f,g\rbrace$ for a regular end. Such differentials have poles of order at most $2$, and if we consider coordinates so that $g$ is expressed as $g(z)=z^n$ then $\sigma$ must have an expansion $\left(\frac{k^2-n^2}{2z^2} +\frac{a}{z} +h.o.t.\right)dz^2$ for some integer $k\geq 2$. Here we have a regular end (as $\omega(f,g)$ has a pole singularity) so that $S\lbrace f,g\rbrace$ has a pole. Hence we will say that a regular end of a Bryant surface is \emph{type-2} if $S\lbrace f,g\rbrace$ has a pole.

\begin{remark}
As we are interested in surfaces of finite-type and since Definition \ref{defi:Small} already encodes ends through the use of meromorphic functions and forms, from now on we will denote by $\overline{\Sigma}$ a compact Bryant surface (possibly branched) with ends points, and we will denote by $\Sigma$ to the surface obtained by removing the ends. 
\end{remark}

As a Corollary of this discussion, we have the following characterizations.

\begin{cor}\label{cor:Bryantproperties}
Consider $\Sigma$ a Bryant surface with no branched points. Then
\begin{enumerate}
    \item $\Sigma$ is immersed in $\mathbb{H}^3$ if and only if $\rho$ is trivial.
    \item If $Im(\rho)$ is a discrete subgroup of $\PSL$ then $\Sigma$ is immersed in the hyperbolic manifold $\HH/Im(\rho)$. Moreover, if $\rho$ is faithful then $\Sigma$ is incompressible.
    \item Assume for this item that $\Sigma$ could be branched. The $\PSL$ valued form $F^{-1}dF$ is well-defined on $\Sigma$ if and only if $h$ is trivial (or equivalently, if $g$ is well-defined on $\Sigma$). We will distinguish this class of Bryant surfaces in the next definition.
\end{enumerate}
\end{cor}

\begin{defi}\label{defi:framed}
    We say that a (branched) Bryant surface is \emph{framed} if the $\PSL$ valued form $F^{-1}dF$ is well-defined on $\Sigma$.
\end{defi}

\subsubsection{Examples}\label{subsec:examples}

As a simple example, any horosphere can be conjugated so that $F = \begin{pmatrix} 1 & z\\  0 & 1 \end{pmatrix}$. In this case the form $F^{-1}dF$ is equal to $\begin{pmatrix} 0 & dz\\  0 & 0 \end{pmatrix}$, which one can easily verify that is invariant by the isometries preserving the horosphere. Hence any Bryant surface modeled locally by an horosphere is framed.

In \cite[Example 2]{Bryant87} Bryant gives a family of Bryant surfaces called the catenoid cousins, which in fact via Lawson's correspondence they match a family of Euclidean catenoids. The family of maps $F$ for $\overline{\Sigma} = \overline{\mathbb{C}}, \Sigma=\mathbb{C}^*$ is given by
\[ F(z) = \frac{1}{\sqrt{2\mu+1}} \begin{pmatrix} (\mu+1)z^\mu & \mu z^{-(\mu+1)}\\  \mu z^{\mu+1} & (\mu+1)z^{-\mu}\end{pmatrix},
\]
where $\mu\neq0$ is a real parameter satisfying $\mu>-\frac12$. One can compute
\[ F^{-1}dF(z) = \frac{\mu(\mu+1)}{2\mu+1} \begin{pmatrix} z^{-1} & -z^{-2\mu-2}\\  z^{2\mu} & -z^{-1}\end{pmatrix}dz.
\]
Hence a catenoid cousins is framed if and only if $2\mu$ is a positive integer.

Families of examples of Bryant surfaces have been constructed in the literature through \emph{duality} (see for instance \cite{UmeharaYamada96}, \cite{RossmanUmeharaYamadaIrreducible}). In the context of Definition \ref{defi:Small}, duality involves interchanging the Gauss maps $f$ and $g$ (or equivalently, by taking the inverse of the null map $F=\omega$). In order to obtain a dual surface that is also equivariant, one needs $f$ (or equivalently, the surface itself) to be equivariant through a representation of $\pi_1(\Sigma)$ into $\SU$. Hence by the first item of Corollary \ref{cor:Bryantproperties} we see that a $\SU$ equivariant surface $\Sigma_1$ is immersed if and only if its dual $\Sigma_2$ is framed.

We can use that in order to be framed we can equivalently verify that $g$ is well-defined in $\Sigma$, which is useful if the Bryant-Weierstrass data of the surface is given explicitly. For instance, in \cite{UmeharaYamada93} there are families of examples for which we can easily determine when they are framed. In \cite[Example 7.2]{UmeharaYamada93} we have $g(z)=\left(\frac{z-1}{z+1} \right)^\mu \left(\frac{z-\mu}{z+\mu}\right)$ defined on $\mathbb{C}\setminus\lbrace -1, 1\rbrace$ with $\mu\in \mathbb{R}\setminus\lbrace \pm1, 0\rbrace$. In this case the surface is framed if and only if $\mu\in \mathbb{Z}\setminus\lbrace\pm1,0 \rbrace$. Similarly for \cite[Example 7.3]{UmeharaYamada93} we have $g(z) = z^\mu \frac{z^m+a}{az^m+1}$ on $\mathbb{C}\setminus\lbrace0,\zeta,\zeta^2,\ldots,\zeta^{m}\rbrace$, where $m$ is an integer, $\zeta$ a $m$-root of $-1$, $\mu \in \mathbb{R}\setminus \lbrace0, \pm1, \pm m\rbrace$ and $a=\frac{\mu+m}{\mu-m}$. Hence these surfaces are framed if and only if $\mu\in\mathbb{Z}\setminus \lbrace0, \pm1, \pm m\rbrace$.

One can use Corollary \ref{cor:Bryantproperties} to produce general families of examples of Bryant surface in hyperbolic manifolds. Consider a projective surface $\Sigma$, that is, a surface with local charts into $\overline{\CC}$ so that the change of charts are restriction of maps in $\PSL$. Take $f:\widetilde{\Sigma}\rightarrow\overline{\CC}$ to be the developing map of this projective structure, which is $\rho$-equivariant for some homomorphism $\rho:\pi_1(\Sigma)\rightarrow\PSL$. Then for any $g:\Sigma\rightarrow\overline{\CC}$ we have that $\omega=\omega(f,g)$ defines a framed branched Bryant surface.

Assume that the projective structure of $\Sigma$ comes from an identification of $\widetilde{\Sigma}$ with an open set of $U\subset\CC$ so that $\pi_1(\Sigma)$ acts faithfully and properly discontinuous in $U$ by restriction of elements of $\PSL$. Following \cite{Epstein84} $\widetilde{\Sigma}$ is in 1-to-1 correspondence with (singular) $\pi_1(\Sigma)$ invariant metrics in $U$ so that its principal curvatures at infinity $k^*_{1,2}$ satisfy $k^*_1+k^*_2 + 2k^*_1k^*_2=0$. This is because the principal curvatures of $\widetilde{\Sigma}$ can be computed as $\frac{1-k^*_1}{1+k^*_1}$, $\frac{1-k^*_2}{1+k^*_2}$.

If $\Sigma$ is a component of the conformal boundary at infinity of a hyperbolic $3$-manifold $M$, then the Bryant surface defined by $\omega$ is a framed Bryant surface in $M$ parallel to the boundary component corresponding to $\Sigma$. As by the main results of \cite{Kahn-Markovic}, \cite{KahnWright} complete finite-volume hyperbolic manifolds contain infinitely many quasi-Fuchsian subgroups, and as each of those subgroups has associated two projective structures at infinity, then it follows that $M$ contains infinitely many finite-type, $\pi_1$-injective, branched, framed Bryant surfaces.  

In order to continue with the demostration of Theorem \ref{thm:Lawsonmirror}, let us now describe how the $1$-to-$1$ correspondence between Euclidean minimal surfaces and Bryant surfaces looks like under Definition \ref{defi:Bryant} and Definition \ref{defi:Small}. For an intrinsic minimal surface $\Sigma$ in $\RR^3$, the Weierstrass representation can be expressed in terms of a meromorphic function $g:\widetilde{\Sigma}\rightarrow\overline{\CC}$ equivariant with respect to a representation $h:\pi_1(\Sigma)\rightarrow\SU$ and a holomorphic $1$-form $\eta$ in $\widetilde{\Sigma}$ so that $\eta\dd g$ defines an $\pi_1(\Sigma)$ invariant $2$-form in $\Sigma$. Then one can construct a branched minimal immersion $x:\widetilde{\Sigma}\rightarrow\RR^3$ so that $g$ is the Gauss map by taking
\begin{equation}\label{eq:minimalimmersion}
    x(z) := {\rm Re}\left(\int_{z_0}^z\frac14(1+g^2)\eta, \int_{z_0}^z \frac{i}4(1-g^2)\eta, \int_{z_0}^z \frac12g\eta\right)
\end{equation}
Such immersion is equivariant by some representation of $\pi_1(\Sigma)$ into ${\rm Isom}^+(\RR^3)$. For this immersion, beyond computing $g$ as the Gauss map, we can compute as well

\begin{enumerate}
    \item The induced metric $ds^2$ as $ds^2 = \frac14(1+|g|^2)^2|\eta|^2$
    \item The holomorphic quadratic differential $\sigma$ so that ${\rm Re}\sigma$ is the second fundamental form as $\sigma = \eta\dd g$.
\end{enumerate}
Observe that the metric $ds^2 = \frac14(1+|g|^2)^2|\eta|^2 = \frac14 \frac{(1+|g|^2)^2}{|dg|^2}|\eta d g|^2$ is a well-defined metric in $\Sigma$ since $g$ is equivariant by a $\SU$ representation and $\eta d g$ defines a $2$-form in $\Sigma$.

Then we can realize $ds^2,\sigma$ as the metric and traceless second fundamental form of $\Sigma$ as a Bryant surface by solving $f:\widetilde{\Sigma}\rightarrow\overline{\CC}$ so that $S\lbrace f,g\rbrace = -\sigma$. Again, as $\sigma$ is a well-defined holomorphic for in $\Sigma$ then $f$ is $\rho$-equivariant for some representation $\rho:\pi_1(\Sigma)\rightarrow\PSL$.

Conversely, given a Bryant surface $\Sigma$ in the sense of Definition \ref{defi:Small} we can realize $\Sigma$ as a intrinsice minimal surface in $\RR^3$ by taking $g$ and $\eta=-S\lbrace f,g\rbrace( d g)^{-1}$ as its Weierstrass data.

In particular, a Bryant surface is framed if and only if the branched immersion (\ref{eq:minimalimmersion}) is equivariant by a representation of $\pi_1(\Sigma)$ into translations of $\RR^3$, which can be identified with $\RR^3$ itself.

Another way to understand framed Bryant/minimal surfaces is as follows. Bryant \cite[Equation 1.9]{Bryant87} describes the $\PSL$ valued form $F^{-1} d F$ as

\begin{equation}\label{eq:harmonicformsPSL2C}
F^{-1} d F = F^*\begin{pmatrix} (w^3+iw^2_1) & (w^1-w^1_3) + i(w^2-w^2_3)\\  (w^1+w^1_3) - i(w^2+w^2_3) & -(w^3+iw^2_1) \end{pmatrix}
\end{equation}
where $(w^i_j)$ are the connection forms of $\PSL$ as a submanifold of $(\mathbb{L}^4)^4$, where $\mathbb{L}^4$ is the $4$-dimensional Minkowski space-time. The expression 

\[
\begin{pmatrix} (w^3+iw^2_1) & (w^1-w^1_3) + i(w^2-w^2_3)\\  (w^1+w^1_3) - i(w^2+w^2_3) & -(w^3+iw^2_1) \end{pmatrix}
\]
defines a $\mathfrak{sl}_2(\CC)$ valued form in $\PSL$, which is invariant by left multiplication. In particular, $\Sigma$ is framed if and only if the harmonic forms $w^1, w^2, w^3$ (and subsequently their harmonic duals $-w^2_3,w^1_3,w^2_1$) are well-defined in $\Sigma$. As $\omega=F$, we can compute
\[
\begin{split}
w^1 &= -\frac14 {\rm Re}\left( (1-g^2)S\lbrace f,g\rbrace d g^{-1}\right),\\
w^2 &= -\frac{i}{4}{\rm Re}\left( (1+g^2)S\lbrace f,g\rbrace d g ^{-1}\right),\\
w^3 &= -\frac12 {\rm Re}\left(gS\lbrace f,g\rbrace dg^{-1}\right).
\end{split}
\]
By (\ref{eq:minimalimmersion}) then $w^1=dx_1, w^2=dx_2, w^3=dx_3$. Hence a Bryant/minimal surface $\Sigma$ is framed if and only if these three harmonic forms are well-defined on $\Sigma$, which is equivalent to the associated $\mathfrak{sl}_2(\CC)$ or $\CC^3$ valued forms to be well-defined on $\Sigma$.

Finally, let us discuss how two-sidedness translate under Lawson's correspondence for Bryant and minimal surfaces. If a Euclidean minimal surface is two-sided then the Gauss map $g$ is well-defined as an equivariant map along $\widetilde{\Sigma}$, while the quadratic holomorphic differential $S\lbrace f,g\rbrace$ is well-defined in $\Sigma$. Hence by solving the Schwarzian equation the map $f$ is well-defined as an equivariant map in $\widetilde{\Sigma}$, which implies that the surface is two-sided as a Bryant surface. Similarly, if we have a two-sided Bryant surface then now $f$ is well-defined equivariantly along $\widetilde{\Sigma}$ and $S\lbrace g, f\rbrace = -S\lbrace f,g \rbrace$ is well-defined in $\Sigma$, so we can check that $g$ is well-defined by solving the Schwarzian equation.

\subsection{Schwarzian derivative}\label{subsec:Schwarzian derivative}

Let $\Sigma$ be a Riemann surface, $g:\widetilde{\Sigma}\rightarrow\overline{\CC}$ be a $\PSL$ equivariant meromorphic function and let $\sigma$ be a holomorphic quadratic differential. For a given point $z_0\in \widetilde{\Sigma}$ lets take local chart identifying $z_0$ with $0$ so that $g(z)=g(z_0) + z^n$ (assuming without loss of generality that $z_0$ is not a pole). If in this chart $\sigma=\sigma(z)dz^2$, then we are interested in solving near $0$ the equation
\begin{equation}\label{eq:Sfg1}
    S\lbrace f, g(z_0)+z^n \rbrace = S\lbrace f, z^n \rbrace = -\sigma(z)dz^2,
\end{equation}
where the first equality follows because translations lie in $\PSL$ and the Schwarzian stays the same by post-composition of $g$ with elements of $\PSL$.

It is a simple calculation to find that
\begin{equation}
    S\lbrace f, z^n \rbrace = S\lbrace f, z \rbrace + \frac{n^2-1}{2z^2}dz^2,
\end{equation}
so (\ref{eq:Sfg1}) is equivalent to
\begin{equation}\label{eq:Sfg2}
    S\lbrace f, z \rbrace = \left(-\sigma(z) - \frac{n^2-1}{2z^2} \right)dz^2,
\end{equation}

Solutions to the Schwarzian equation $S\lbrace f, z \rbrace = q(z)dz^2$, where $q$ is holomorphic at $0$ can be find as follows. One considers the linear differential equation
\begin{equation}\label{eq:linearPDE}
p''+ \frac{q}{2}p = 0
\end{equation}
As the system is linear one can consider two linearly independent solution $p_1,p_2$. Then we can solve $S\lbrace f, z \rbrace = q(z)dz^2$ by taking $f=\frac{p_1}{p_2}$.

In general if we consider $q$ with a pole at $0$, the solutions $p_1, p_2$ and the quotient $f$ are now multivalued functions in the punctured disk. This is because we solve the problem $S\lbrace f, z \rbrace = q(z)dz^2$ in the strip covering of the punctured disk, but the solution does not need to be invariant by deck transformations. Nevertheless, we will see how if $q$ has a pole of at most $2$ with $z^{-2}$ coefficient equal to $\frac{n^2-1}{2}$ for some positive integer $n$, then even though the solutions $p_1, p_2$ of (\ref{eq:linearPDE}) are multivalued, the quotient $f$ is well-defined around $0$. The motivation for what follows is that in the case when $q(z)=-\frac{n^2-1}{2}$ then $p_1(z)=z^{\frac{n+1}{2}}, p_2(z)=z^{\frac{-n+1}{2}}$ are the multivalued solutions to (\ref{eq:linearPDE}) while $f(z)=p_1(z)/p_2(z)=z^n$ is well-defined. 

Consider then the equation
\begin{equation}\label{eq:linearPDE1}
p''+ \left(-\frac{\sigma(z)}{2} - \frac{n^2-1}{4z^2} \right)p = 0
\end{equation}

Let us construct a solution $p_1$ given by

\[p_1(z) = z^{\frac{n+1}{2}}\left(1+\sum_{j\geq1}a_jz^j\right) = z^{\frac{n+1}{2}} + \sum_{j\geq1}a_jz^{\frac{n+1}{2}+j}
\]

Hence
\begin{equation}\label{eq:p1''term}
    p''_1(z) = \frac{1}{z^2}\left(\frac{n^2-1}{4}z^{\frac{n+1}{2}} + \sum_{j\geq1}a_j\left(\frac{n+1}{2}+j\right)\left(\frac{n-1}{2}+j\right)z^{\frac{n+1}{2}+j}\right)
\end{equation}
while
\begin{equation}\label{eq:qp1term}
    p_1(z)\left(- \frac{n^2-1}{4z^2} -\frac{\sigma(z)}{2} \right) = \frac{1}{z^2} \left(z^{\frac{n+1}{2}} + \sum_{j\geq1}a_jz^{\frac{n+1}{2}+j}\right) \left( - \frac{n^2-1}{4} -\frac{\sigma(z)z^2}{2}\right).
\end{equation}
Then equation (\ref{eq:linearPDE1}) determines the coefficients $a_j$ recursively.

Similarly we can find a solution $p_2$ with an expansion
\[p_2(z) = z^{\frac{-n+1}{2}}\left(1+\sum_{j\geq1}b_jz^j\right) = z^{\frac{n+1}{2}} + \sum_{j\geq1}b_jz^{\frac{n+1}{2}+j}
\]

Then the quotient

\[f(z) = p_1(z)/p_2(z) = z^n\frac{(1+\sum_{j\geq1}a_jz^j)}{(1+\sum_{j\geq1}b_jz^j)}
\]
is a well-defined holomorphic function at the origin that satisfies (\ref{eq:Sfg2}). As we can see, if $0$ was a critical point of $g$ then it is a critical point for $f$ with the same multiplicity.

Observe that if $\sigma$ has a $0$ of order $n$ at $0$, then the recursion that finds $p_1,p_2$ shows that $a_j=b_j=0$ for any $1\leq j\leq n$. Hence in this case we have that $f(z) = z^n(1+z^n\zeta(z))$ for some $\zeta$ holomorphic at $0$. Then $\frac{df}{dg} = 1+ z^n\left( 2\zeta(z) + \frac{z}{n}\zeta'(z)\right)$ has also a critical point of at least the same common multiplicity of $f$ and $g$.

Finally, if $\sigma=\sigma(z)dz^2$ has at most a pole of order $2$ at the origin and $\sigma(z)$ can be expressed as
\[\sigma = \frac{k^2-n^2}{2z^2} + h.o.t.
\]
for some positive integer $k$, then (\ref{eq:Sfg2}) becomes
\[S\lbrace f,g\rbrace = \left(-\frac{k^2-1}{2z^2}+h.o.t \right)dz^2
\]
which has a well-defined holomorphic solution $f$ near the origin of multiplicity $k$.

\section{Weighted space}\label{sec:Weighted space}

From now on we refer to $\Sigma$ as a framed surface, meaning either framed Euclidean minimal or framed Bryant. We will also assume from now that besides being framed, $\Sigma$ is finite-type. At each end we will also fix a complex chart identifying the end with the origin, and we say that the end has order $m$ if $|z|^{2m}ds^2$ extends to a complete metric at the origin. We will fix a positive smooth function $u$ in $\Sigma$ so that at an end of order $m$ we have that $u(z)=\frac{|z|^{m-1}}{(m-1)\log|z|^{-1}}$ ($u(z)=(\log|z|^{-1})^{-1}(\log(\log|z|^{-1}))^{-1}$ if $m=1$), and we define $L^2_{*}(\Sigma)$ as the space of real valued functions in $\Sigma$ by taking the completion of compactly supported functions on $\Sigma$ with respect to the inner product

\[ \langle f, g \rangle_{L^2_{*}(\Sigma)} = \int_{\Sigma} u^2fgda
\]

Observe that since the metric $ds^2$ and Gaussian curvarture $K$ are given by $ds^2 = \frac14(1+|g|^2)^2|S\lbrace f,g \rbrace|^2|dg|^{-2}$ and $K=-4(1+|g|^2)^{-4}|S\lbrace f,g \rbrace|^{-2}|dg|^{4}$ we have that $|K|\leq C|z|^{2m+\ell}$ for some constant $C>0$, where $\ell\geq0$ is the order of $dg$ at the end. In particular we have that $|K|\leq Cu^2$.

The (weighted) eigenfunction equation with respect to $\langle.,.\rangle_{L^2_{*}(\Sigma)}$ of the quadratic form $Q(f,f)=\int_{\Sigma} |\nabla f|^2 + 2Kf^2 da$ is given by
\begin{equation}
    \Delta f - 2Kf + \lambda u^{2}f = 0
\end{equation}

Following \cite[Proposition 8]{ChodoshMaximo} we state

\begin{prop}\label{prop:weightedspace}
Suppose $\Sigma$ has index $k$. There exists a $k$-dimensional subspace $W$ in $L^2_{*}(\Sigma)$ with a $L^2_{*}(\Sigma)$ orthonormal basis of eigenfunctions $f_1,\ldots,f_k$. The associated eigenvalues $\lambda_1,\ldots,\lambda_k$ are all negative, and for any $\phi$ in $C^\infty_0(\Sigma)\cap W^\perp$ (where $W^\perp$ denotes the $L^2_{*}(\Sigma)$ orthogonal complement of $W$) we have that $Q(\phi,\phi)\geq0$.
\end{prop}
\begin{proof}
As opposed to \cite{ChodoshMaximo} (and in similarity with \cite{Fischer-Colbrie}) we work intrinsically, although we follow the same strategy. Let us further assume that $\Sigma$ restricted to the union of the complex charts we have fixed for the ends is stable. For each end of order $m$ consider the set of $|z|\leq R^{1/(1-m)}$ ($|z|\leq e^{-R}$ for $m=1$) and denote by $B_R$ the complement of their union. The set $B_R$ coarsely represents a ball of radius $R$. It is easy to see that for $R$ sufficiently large we have $u^{-1}(-\infty, (R\log R)^{-1}) = \Sigma\setminus B_R$, so in particular $u(z)\geq (R\log R)^{-1}$ for any $z\in B_R$. Similar to \cite[Proposition 1]{Fischer-Colbrie} we fix smooth $\xi:\mathbb{R}\rightarrow\mathbb{R}$ and take a log-cutoff $\psi_R(z)$ defined near each end of order $m$ as $\psi_R(z) = 2-\frac{(m-1)\log|z|^{-1}}{\log R}$ ($\psi_R(z) = 2-\frac{\log(\log|z|^{-1})}{\log R}$ for $m=1$) to define $\eta=1-\xi\circ\psi_R$. This function satisfies
\begin{enumerate}
    \item $\eta\equiv 0$ on $B_R$
    \item $\eta\equiv 1$ on $\Sigma\setminus B_{R^2}$.
    \item $|\nabla\eta|\leq Cu$, $|\Delta\eta|\leq Cu^2$ on $B_{R^2}\setminus B_R$, for some for some constant $C>0$ independent of $R$. In particular we have that $|\nabla\eta|\leq \frac{C}{R\log R}$, $|\Delta\eta|\leq \frac{C}{R^2\log^2R}$. 
\end{enumerate}

Hence for any $\phi\in C^\infty_0(\Sigma)$ we can write the stability inequality for $\eta\phi$ as

\begin{equation}\label{eq:stability}
    -\int_\Sigma 2K(\eta\phi)^2 \leq \int_\Sigma |\nabla(\eta\phi)|^2 = \int_\Sigma \eta^2|\nabla\phi|^2 +2\eta\phi\langle\nabla\eta,\nabla\phi\rangle + \phi^2|\nabla\eta|^2.
\end{equation}

Adding $Q(\phi,\phi) =\int_{\Sigma} |\nabla\phi|^2+2K\phi^2$ to both sides and using Cauchy-Schwartz we obtain

\begin{equation}
\begin{split}
    \int_\Sigma (1-\eta^2)(|\nabla\phi|^2+2K\phi^2) &\leq Q(\phi,\phi) + \int_\Sigma \eta^2|\nabla\phi|^2 + 2\phi^2|\nabla\eta|^2\\
    &\leq 2Q(\phi,\phi) + \int_\Sigma  2\phi^2|\nabla\eta|^2 -2K\phi^2.
\end{split}
\end{equation}

Using the conditions on $\eta$ we get
\begin{equation}\label{eq:gradienbound}
    \int_{B_R} |\nabla\phi|^2 \leq 2Q(\phi,\phi) + \left(\frac{2C^2}{R^2\log^2 R} + 2\sup_{B_{R^2}}|K|\right) \int_{B_{R^2}} \phi^2
\end{equation}

By defining $C_R:= \left(\frac{2C^2}{R\log^2 R} + \sup_{B_{R}}|K|\right)$ and using that $u\leq(R\log R)^{-1}$ in $B_R$, we can conclude that for $R$ sufficiently large
\begin{equation}
    -C_R (R\log R)^{2}\Vert\phi\Vert^2_{L^2_{*}(B_{R}))}\leq-C_R\int_{B_{R}} \phi^2\leq Q(\phi,\phi),
\end{equation}

Taking $R$ even larger if needed, we can assume that for any $R'\geq R$ we have $k=\idx(\Sigma) = \idx(B_{R'})$. Let then $\lbrace f_{1,R'},\ldots f_{k,R'}\rbrace$ and $\lbrace \lambda_{1,R'},\ldots \lambda_{k,R'}\rbrace$ denote the $L^2_{*}$ eigenfunctions and eigenvalues obtained by minimizing the Rayleigh quotient $Q(\phi,\phi)/\Vert\phi\Vert^2_{L^2_{*}(B_R)}$. Since $\max \lbrace \lambda_{1,R'},\ldots \lambda_{k,R'}\rbrace$ is decreasing in $R'$, we have that there exists $\epsilon_0>0$ so that $\lbrace \lambda_{1,R'},\ldots \lambda_{k,R'}\rbrace$ belong to the interval $]-\infty,-\epsilon_0]$ for any $R'>R$.

Extending $\phi=f_{i,R'}$ as $0$ and plugging it in \eqref{eq:stability}, we proceed as in \cite{ChodoshMaximo} by changing our choice of $\eta$ so that 

\begin{enumerate}
    \item $\eta\equiv 0$ on $B_R$
    \item $\eta\equiv 1$ on $\Sigma\setminus B_{R^\beta}$.
    \item $\eta(z) = \frac{1}{\log\beta}\log\left(\frac{(m-1)\log|z|^{-1}}{\log R}\right)$ ($\eta(z) = \frac{1}{\log\beta}\log\left(\frac{\log(\log|z|^{-1})}{\log R}\right)$ for $m=1$) for $B_{R^\beta}\setminus B_{R}$  near an end of order $m$.
    \item $|\nabla\eta|\leq \frac{Cu(z)}{\log\beta}$ on $B_{R^\beta}\setminus B_R$ for some $C>0$ independent of $\beta, R$.
\end{enumerate}
to obtain

\begin{equation}\label{eq:altstabilityapplication}
    \epsilon_0\Vert \eta f_{i,R'}\Vert^2_{L^2_{*}(\Sigma)} \leq -\lambda_{i,R'}\Vert \eta f_{i,R'}\Vert^2_{L^2_{*}(\Sigma)} \leq \int_{\Sigma} (f_{i,R'})^2|\nabla\eta|^2\leq \frac{C^2\int_{B_{R^{\beta}}} u^2 f_{i,R'}^2}{\log^2\beta} .
\end{equation}

Combining \eqref{eq:altstabilityapplication} with $0\leq \eta \leq 1$ and $\eta\equiv 1$ on $\Sigma\setminus B_{R^\beta}$ we obtain

\begin{equation}\label{eq:noL2deltaloss}
    \int_{\Sigma\setminus B_{e^{\beta}R}}u^{2}(f_{i,R'})^2 \leq \Vert \eta f_{i,R'}\Vert^2_{L^2_{*}(\Sigma)} \leq \frac{C^2 \Vert f_{i,R'}\Vert^2_{L^2_{*}(\Sigma)}}{\epsilon_0\log^2\beta} = \frac{C^2}{\epsilon_0 \log^2\beta},
\end{equation}
which implies that $f_{i,R'}$ does not lose mas at infinity by taking arbitrarily large $\beta$.

Also, using \eqref{eq:gradienbound} we get

\begin{equation}\label{eq:W12control}
\begin{split}
    &\int_{B_R} |\nabla f_{i,R'}|^2 \leq \lambda_{i,R'} + C_{R^2}\int_{B_{R^2}}(f_{i,R'})^2 \leq \lambda_{i,R'} + C_{R^2}(4R^4\log^2R) \Vert f_{i,R'} \Vert^2_{L^2_{*}(\Sigma)}\\
    &\Rightarrow \int_{B_R}(f_{i,R'})^2 + |\nabla f_{i,R'}|^2 \leq \lambda_{i,R'} + [R^2\log^2R + C_{R^2}(4R^4\log^2R)] \Vert f_{i,R'} \Vert^2_{L^2_{*}(\Sigma)}
\end{split}
\end{equation}

By equations \eqref{eq:noL2deltaloss} and \eqref{eq:W12control}, a diagonal argument and the compact embeddedness of $W^{1,2}(B_R)$ in $L^2_{*}(B_R)$ we obtain $f_i=\lim_{n\rightarrow\infty}f_{i,R_n}$ for some sequence $R_n\rightarrow+\infty$. By the arguments of \cite[Proposition 2]{Fischer-Colbrie} the functions $\lbrace f_1,\ldots, f_k\rbrace$ are as desired.
\end{proof}

Given a finite-type framed surface $\Sigma$ we define its fundamental divisor as $D_\Sigma = \sum_P n_P.P$, where $n_P=-m$ if $P$ is an end of order $m$, $n_P=m$ if $P$ is a branching point of order $m$ and $n_P=0$ otherwise. In the next Proposition we characterize the $L^2_{*}$ integrability of a meromorphic 1-form by comparing its divisor against the fundamental divisor $D$.

\begin{prop}\label{prop:L2deltaharmonics}
Let $\Sigma$ be a compact and framed surface. Let $\varsigma$ be a meromorphic $1$-form on $\Sigma$. Then $D_\varsigma\geq D_\Sigma$ if and only if $\varsigma$ is in $L^2_{*}(\Sigma)$.
\end{prop}
\begin{proof}
Consider an end $P$ of order $m$ with its fixed complex chart to the origin and let $m'$ be the order of $\varsigma$ at $P$. As  the square norm of a $1$-form times the area form is a conformal invariant, $L^2_{*}$ integrability of $\varsigma$ near $P$ reduces to prove

\[\Vert z^{m'}dz\Vert^2_{L^2_{*}(\mathbb{D}\setminus\lbrace0\rbrace} = \int_{\mathbb{D}\setminus\lbrace0\rbrace} u^{2}|z|^{2m'}dxdy<\infty
\]

If $m>1$ this reduces to $\int_{\mathbb{D}\setminus\lbrace0\rbrace} \frac{|z|^{(2m-2)+2m'}}{\log^2|z|^{-1}}dxdy<\infty$ which is true if and only if $(2m-2)+2m'\geq -2$.

For $m=1$ this reduces to $\int_{\mathbb{D}\setminus\lbrace0\rbrace)} \frac{|z|^{2m'}}{\log^2|z|^{-1}\log^2(\log|z|^{-1})}dxdy<\infty$ which is finite if and only if $m'\geq -1$.

\end{proof}

Next, we will see that weighted eigenfunction have square integrable gradient.

\begin{prop}\label{prop:L^2gradient}
Let $f\in L^2_{*}(\Sigma)$ be a smooth function satisfying the weighted eigenfunction equation. then $|\nabla f|\in L^2(\Sigma)$.
\end{prop}
\begin{proof}
Following \cite{ChodoshMaximo} we take $\eta$ as in the first part of the proof of Proposition \ref{prop:weightedspace}. Our goal will be to prove that $\int_{\Sigma}\eta^2|\nabla f|^2$ is uniformly bounded.

\[
\begin{split}
    \int_{\Sigma}\eta^2|\nabla f|^2 &= -\int_{\Sigma} \eta^2 f\Delta f - \frac12 \int_{\Sigma} \nabla \eta^2 .\nabla f^2\\
    &=-\int_{\Sigma}2K\eta^2f^2 + \lambda \int_{\Sigma}\eta^2f^2u^2 - \frac12 \int_{\Sigma} \nabla \eta^2 .\nabla f^2
\end{split}
\]

As $|K|\leq Cu^2$ and $f\in L^2_{*}(\Sigma)$, we only have to check that $\int_{\Sigma} \nabla \eta^2 .\nabla f^2$ is uniformly bounded.

\begin{equation}
    - \int_{\Sigma} \nabla \eta^2 .\nabla f^2 = \int_{\Sigma} f^2\Delta\eta^2 = 2\int_{\Sigma}f^2\eta\Delta\eta + 2\int_{\Sigma} f^2|\nabla\eta|^2
\end{equation}

We bound the second term by

\begin{equation}
\begin{split}
    \int_{\Sigma} f^2|\nabla\eta|^2 &\leq \frac{C^2}{R^4\log^2 R}\int_{B_{R^2}\setminus B_R}f^2\\
    &\leq \frac{4C^2R^4\log^2R}{R^4\log^2 R} \Vert f\Vert^2_{L^2_{*}} = 4C^2\Vert f\Vert^2_{L^2_{*}}
\end{split}
\end{equation}

We bound the first term by

\begin{equation}
\begin{split}
    \int_{\Sigma} f^2\eta\Delta\eta &\leq \frac{C^2}{R^4\log^2 R}\int_{B_{R^2}\setminus B_R}f^2\\
    &\leq \frac{4C^2R^4\log^2R}{R^4\log^2 R} \Vert f\Vert^2_{L^2_{*}} = 4C^2\Vert f\Vert^2_{L^2_{*}}
\end{split}
\end{equation}
\end{proof}

The next Lemma (which concerns the square integrability of the gradient of a harmonic form in $L^2_{*}$) follows as done for \cite[Lemma 4.2]{ChodoshMaximo}, so we omit the proof here.

\begin{lem}\label{lem:L2gradientw}
Let $w\in L^2_{*}(\Sigma)$ be harmonic form. Then $|\nabla w|\in L^2(\Sigma)$.
\end{lem}

\section{Proof of Theorem \ref{mainthm:index}}\label{sec:ProofA}
We follow the same approach as \cite{Ros06} and \cite{ChodoshMaximo}. Given a harmonic $1$-form $w$ we define the vector field

\[X_w:=(\langle w,dx_1\rangle, \langle w,dx_2\rangle, \langle w,dx_3\rangle)
\]

By \cite[Lemma 1]{Ros06} we have that

\[\Delta X_w -2KX_w = 2\langle \nabla w,{\rm Re}(\sigma)\rangle g
\]
where we recall that $K$ is the Gaussian curvature, ${\rm Re}(\sigma)$ is the second fundamental form and $g$ is the Gauss map. Moreover, $\langle \nabla w,{\rm Re}(\sigma)\rangle\equiv 0$ if and only if $w\in {\rm span}\lbrace *dx_1, *dx_2,*dx_3\rangle$

As in \cite{ChodoshMaximo}, given two vector fields $X,Y$ in $\Sigma$ we will denote the sum $\sum_{i=1}^3Q(X_i,Y_i)$ by $Q(X,Y)$. 

Given $k={\rm Ind}(\Sigma)$, Proposition \ref{prop:weightedspace} says that we have $L^2_*$ eigenfunctions $f_1,\ldots,f_k$ with span $W\subset L^2_*(\Sigma)$ so that for any $\phi\in C^\infty_0(\Sigma)\cap W^\perp$ we have that $Q(\phi,\phi)\geq 0$. By Proposition \ref{prop:L2deltaharmonics} we have that the space $V=L^2_*(\Sigma)\cap \mathcal{H}^1(\Sigma)$ has dimension equal to $2h^1(D)$, where $D$ is the divisor defined before Proposition \ref{prop:L2deltaharmonics}.

Our goal will be to show that if $w\in V$ satisfies $X_w\in W^\perp$, then $\langle \nabla w,{\rm Re}(\sigma)\rangle\equiv 0$, and subsequently $w\in {\rm span}\lbrace *dx_1, *dx_2,*dx_3\rangle$.

Suppose then that we have $w\in V$ satisfying $X_w\in W^\perp$. We will show that for any compactly supported vector field $Y\in W^\perp$, we have that $Q(X_w,Y)=0$. Take then $R$ sufficiently large so that $B_R$ contains the support of $Y$, and define

\[X_t :=\eta(X_w + tY + f_1\overrightarrow{c_1} + \ldots + f_k\overrightarrow{c_k})
\]
where $\eta=\eta(R)$ is the first test function used in the proof of Proposition \ref{prop:weightedspace}, and the vectors $\overrightarrow{c_1}, \ldots, \overrightarrow{c_k}$ are chosen so that $X_t\in W^\perp$. As $\eta Y= Y$ and $Y\in W^\perp$, we have then that for any $1\leq j\leq k$

\[\int_{\Sigma} \eta(X_w + f_1\overrightarrow{c_1} + \ldots + f_k\overrightarrow{c_k})f_ju^2 = 0
\]

As in \cite{Ros06}, \cite{ChodoshMaximo} from the stability inequality we have that

\begin{equation}\label{eq:mainQXY}
\begin{split}
    Q(X_w,Y)^2 &\leq Q(Y,Y)\times\\&\Bigg(Q(\eta X_w, \eta X_w) + 2\sum_{j=1}^k Q(\eta X_w, \eta f_j\overrightarrow{c_j}) + \sum_{j,\ell=1}^k Q(\eta f_j\overrightarrow{c_j}, \eta f_\ell\overrightarrow{c_\ell}) \Bigg)
\end{split}
\end{equation}

As in \cite{ChodoshMaximo} we have that we can express $Q(\eta X_w, \eta X_w)$ as
\begin{equation}
\begin{split}
    Q(\eta X_w, \eta X_w) &= \int_{\Sigma} |\nabla\eta|^2|\eta X_w|^2.
\end{split}
\end{equation}
Given that $|\nabla \eta|\leq \frac{C|z|^{m-1}}{\log R}$ we have that

\begin{equation}
\begin{split}
    \int_{\Sigma} |\nabla\eta|^2|\eta X_w|^2 &\leq  C\int_{B_{R^2}\setminus B_R} |X_w|^2|\nabla\eta|^2\\
    &\leq C\int_{B_{R^2}\setminus B_R} |X_w|^2u^2
\end{split}
\end{equation}
goes to $0$ as $R\rightarrow+\infty$, since $X_w\in L^2_*(\Sigma)$.

For the second term in \eqref{eq:mainQXY} we obtain as in \cite{ChodoshMaximo}
\begin{equation}\label{eq:secondtermQXY}
\begin{split}
    Q(\eta X_w, \eta f_j\overrightarrow{c_j}) &= \lambda_j \int_{\Sigma}(\eta)^2 \langle X_w, \overrightarrow{c_j}\rangle f_j u^2 - \int_{\Sigma} \eta\Delta\eta \langle X_w, \overrightarrow{c_j}\rangle f_j \\&-2 \int_{\Sigma} \eta \langle X_w, \overrightarrow{c_j}\rangle \langle\nabla\eta,\nabla f_j\rangle
\end{split}
\end{equation}
The first term in \eqref{eq:secondtermQXY} goes to $0$ by dominated convergence as $R\rightarrow+\infty$, since $X_w, f_j\in L^2_*(\Sigma)$.

For the middle term in \eqref{eq:secondtermQXY} we see that
\begin{equation}\label{eq:2btermQXY}
\begin{split}
    \Bigg|\int_{\Sigma} \eta\Delta\eta \langle X_w, \overrightarrow{c_j}\rangle f_j\Bigg| &\leq|\overrightarrow{c_j}|\int_{B_{R^2}\setminus B_R}|X_w||f_j| |\Delta\eta| \\&\leq C\int_{B_{R^2}\setminus B_R}|X_w||f_j| u^2
\end{split}
\end{equation}
goes to $0$ as $R\rightarrow+\infty$, since $X_w, f_j\in L^2_*(\Sigma)$. 

For the third term in \eqref{eq:secondtermQXY} apply Cauchy-Schwartz to get
\begin{equation}\label{eq:2ctermQXY}
\begin{split}
    \Bigg|\int_{\Sigma} \eta \langle X_w, \overrightarrow{c_j}\rangle \langle\nabla\eta,\nabla f_j\rangle\Bigg| \leq |\overrightarrow{c_j}| \left(\int_{\Sigma} |\nabla\eta|^2|\eta X_w|^2\right)^{1/2} \left(\int_{\Sigma}|\nabla f_j|^2 \right)^{1/2},
\end{split}
\end{equation}
which goes to $0$ as $R\rightarrow+\infty$ as in \ref{eq:2btermQXY}, since $X_w, f_j\in L^2_*(\Sigma)$. Thus \eqref{eq:secondtermQXY} goes to $0$ as $R\rightarrow+\infty$.

Mirroring again \cite{ChodoshMaximo} we obtain for the third term in \eqref{eq:mainQXY}

\begin{equation}\label{eq:thirdtermQXY}
\begin{split}
    Q(\eta f_j\overrightarrow{c_j}, \eta f_\ell\overrightarrow{c_\ell}) =  \frac12(\lambda_j+\lambda_\ell)\langle\overrightarrow{c_j}, \overrightarrow{c_\ell}\rangle \int_{\Sigma}\eta^2f_jf_\ell u^2 + \langle \overrightarrow{c_j}, \overrightarrow{c_\ell}\rangle \int_{\Sigma} |\nabla\eta|^2f_jf_\ell.
\end{split}
\end{equation}
Since $f_j, f_\ell\in L^2_*(\Sigma)$ and $\overrightarrow{c_j}, \overrightarrow{c_\ell}\xrightarrow{R\rightarrow+\infty} 0$, we have that \eqref{eq:thirdtermQXY} goes to $0$ as $R\rightarrow+\infty$.

Following \cite{ChodoshMaximo} for any smooth vector field $Y\in W^\perp$ and $R$ sufficiently large, we can take vector $\overrightarrow{c_j}$ so that

\[Y_R = \eta(Y+c_1f_1+\ldots c_kf_k)
\]
belong to $W^\perp$. By the previous discussion we have that

\begin{equation}
\begin{split}
    0=Q(X_w,Y_R) = -\int_{\Sigma}\langle \Delta X_w -2KX_w , Y_R\rangle =-2 \int_{\Sigma}\langle \nabla w,{\rm Re}(\sigma)\rangle\langle N , Y_R\rangle
\end{split}
\end{equation}

\begin{equation}
    0=\int_{\Sigma}\langle \nabla w,{\rm Re}(\sigma)\rangle\langle N , Y\rangle
\end{equation}

From Proposition \ref{prop:weightedspace} we get for any $\overrightarrow{\alpha}\in\mathbb{R}^3$ and $f_j$ eigenfunction
\begin{equation}
    0=\int_{\Sigma}\langle \nabla w,{\rm Re}(\sigma)\rangle\langle N , f_j\overrightarrow{\alpha}\rangle
\end{equation}
Hence $\langle\nabla w,{\rm Re}(\sigma)\rangle\equiv0$, which in turn implies $w\in {\rm span}\lbrace*dx_1, *dx_2, *dx_3 \rbrace$ as desired.

As the space $V=L^2_*(\Sigma)\cap \mathcal{H}^1(\Sigma)$ has dimension $2h^1(D)$ and $X_w\in W^\perp$ is a system of $3k$ equations on $V$ whose kernel has dimension $\leq3$, then $3k\geq 2h^1(D)-3$, or equivalently $k\geq \frac{2h^1(D)-3}{3}$.

\section{Proof of Theorem \ref{mainthm:non-orientable}}\label{sec:ProofB}

In the same manner as \cite{ChodoshMaximo} we have that if $\Sigma$ is one-sided then the involution defined in the double covering $\widehat{\Sigma}$ defines an isomorphism between invariant and anti-invariant harmonic forms, denoted by $\mathcal{H}^1_+(\widehat{\Sigma})$, $\mathcal{H}^1_-(\widehat{\Sigma})$, respectively. In particular each of the spaces $\mathcal{H}^1_+(\widehat{\Sigma})\cap L^2_*(\widehat{\Sigma})$, $\mathcal{H}^1_-(\widehat{\Sigma})\cap L^2_*(\widehat{\Sigma})$ has real dimension $h^1(D)$.

As the forms $w\in \mathcal{H}^1_-(\widehat{\Sigma})$ satisfy that $X_w$ can be used as a test function as in the proof of Theorem \ref{mainthm:index}, we argue as in the previous proof and in parallel to \cite{ChodoshMaximo} to conclude that $k\geq \frac{h_1(D)-3}{3}$.

\bibliographystyle{amsalpha}
\bibliography{mybib}

\providecommand{\bysame}{\leavevmode\hbox to3em{\hrulefill}\thinspace}
\providecommand{\MR}{\relax\ifhmode\unskip\space\fi MR }
\providecommand{\MRhref}[2]{%
  \href{http://www.ams.org/mathscinet-getitem?mr=#1}{#2}
}
\providecommand{\href}[2]{#2}
\begin{thebibliography}{GNY04}

\bibitem[BB00]{BarbosaBerard}
Lucas Barbosa and Pierre B\'{e}rard, \emph{Eigenvalue and ``twisted''
  eigenvalue problems, applications to {CMC} surfaces}, J. Math. Pures Appl.
  (9) \textbf{79} (2000), no.~5, 427--450. \MR{1759435}

\bibitem[Bry87]{Bryant87}
Robert~L. Bryant, \emph{Surfaces of mean curvature one in hyperbolic space},
  no. 154-155, 1987, Th\'{e}orie des vari\'{e}t\'{e}s minimales et applications
  (Palaiseau, 1983--1984), pp.~12, 321--347, 353 (1988). \MR{955072}

\bibitem[Che23]{Chen}
Shuli Chen, \emph{On the index of minimal surfaces with free boundary in a
  half-space}, J. Geom. Anal. \textbf{33} (2023), no.~2, Paper No. 46, 11.
  \MR{4523278}

\bibitem[CM16]{ChodoshMaximo16}
Otis Chodosh and Davi Maximo, \emph{On the topology and index of minimal
  surfaces}, J. Differential Geom. \textbf{104} (2016), no.~3, 399--418.
  \MR{3568626}

\bibitem[CM23]{ChodoshMaximo}
\bysame, \emph{On the topology and index of minimal surfaces {II}}, J.
  Differential Geom. \textbf{123} (2023), no.~3, 431--459. \MR{4584858}

\bibitem[dLR98]{LimaRossman}
Levi~Lopes de~Lima and Wayne Rossman, \emph{On the index of constant mean
  curvature {$1$} surfaces in hyperbolic space}, Indiana Univ. Math. J.
  \textbf{47} (1998), no.~2, 685--723. \MR{1647877}

\bibitem[Eps]{Epstein84}
Charles Epstein, \emph{Envelopes of horospheres and weingarten surfaces in
  hyperbolic 3-spaces.}

\bibitem[FC85]{Fischer-Colbrie}
D.~Fischer-Colbrie, \emph{On complete minimal surfaces with finite {M}orse
  index in three-manifolds}, Invent. Math. \textbf{82} (1985), no.~1, 121--132.
  \MR{808112}

\bibitem[GL86]{GulliverLawson86}
Robert Gulliver and H.~Blaine Lawson, Jr., \emph{The structure of stable
  minimal hypersurfaces near a singularity}, Geometric measure theory and the
  calculus of variations ({A}rcata, {C}alif., 1984), Proc. Sympos. Pure Math.,
  vol.~44, Amer. Math. Soc., Providence, RI, 1986, pp.~213--237. \MR{840275}

\bibitem[GNY04]{GNY}
Alexander Grigor'yan, Yuri Netrusov, and Shing-Tung Yau, \emph{Eigenvalues of
  elliptic operators and geometric applications}, Surveys in differential
  geometry. {V}ol. {IX}, Surv. Differ. Geom., vol.~9, Int. Press, Somerville,
  MA, 2004, pp.~147--217. \MR{2195408}

\bibitem[Gul86]{Gulliver86}
Robert Gulliver, \emph{Index and total curvature of complete minimal surfaces},
  Geometric measure theory and the calculus of variations ({A}rcata, {C}alif.,
  1984), Proc. Sympos. Pure Math., vol.~44, Amer. Math. Soc., Providence, RI,
  1986, pp.~207--211. \MR{840274}

\bibitem[GY03]{Grigor'yanYau}
Alexander Grigor'yan and Shing-Tung Yau, \emph{Isoperimetric properties of
  higher eigenvalues of elliptic operators}, Amer. J. Math. \textbf{125}
  (2003), no.~4, 893--940. \MR{1993744}

\bibitem[Kar19]{Karpukhin19}
Mikhail Karpukhin, \emph{On the {Y}ang-{Y}au inequality for the first {L}aplace
  eigenvalue}, Geom. Funct. Anal. \textbf{29} (2019), no.~6, 1864--1885.
  \MR{4034923}

\bibitem[KM12]{Kahn-Markovic}
Jeremy Kahn and Vladimir Markovic, \emph{Immersing almost geodesic surfaces in
  a closed hyperbolic three manifold}, Ann. of Math. (2) \textbf{175} (2012),
  no.~3, 1127--1190. \MR{2912704}

\bibitem[KUY03]{KokubuUmeharaYamada}
Masatoshi Kokubu, Masaaki Umehara, and Kotaro Yamada, \emph{An elementary proof
  of {S}mall's formula for null curves in {${\rm PSL}(2,{\bf C})$} and an
  analogue for {L}egendrian curves in {${\rm PSL}(2,{\bf C})$}}, Osaka J. Math.
  \textbf{40} (2003), no.~3, 697--715. \MR{2003744}

\bibitem[KW21]{KahnWright}
Jeremy Kahn and Alex Wright, \emph{Nearly {F}uchsian surface subgroups of
  finite covolume {K}leinian groups}, Duke Math. J. \textbf{170} (2021), no.~3,
  503--573. \MR{4255043}

\bibitem[Law70]{Lawson70}
H.~Blaine Lawson, Jr., \emph{Complete minimal surfaces in {$S\sp{3}$}}, Ann. of
  Math. (2) \textbf{92} (1970), 335--374. \MR{270280}

\bibitem[MIP22]{MeeksPerez22}
William~H. Meeks~III and Joaquin Perez, \emph{Geometry of branched minimal
  surfaces of finite index}, 2022.

\bibitem[MR91]{MontielRos}
Sebasti\'{a}n Montiel and Antonio Ros, \emph{Schr\"{o}dinger operators
  associated to a holomorphic map}, Global differential geometry and global
  analysis ({B}erlin, 1990), Lecture Notes in Math., vol. 1481, Springer,
  Berlin, 1991, pp.~147--174. \MR{1178529}

\bibitem[MW14]{MaWh}
Francisco Mart\'{\i}n and Brian White, \emph{Properly embedded, area-minimizing
  surfaces in hyperbolic 3-space}, J. Differential Geom. \textbf{97} (2014),
  no.~3, 515--544. \MR{3263513}

\bibitem[Ros92]{Ross92}
Marty Ross, \emph{Complete nonorientable minimal surfaces in {$\bold R^3$}},
  Comment. Math. Helv. \textbf{67} (1992), no.~1, 64--76. \MR{1144614}

\bibitem[Ros06]{Ros06}
Antonio Ros, \emph{One-sided complete stable minimal surfaces}, J. Differential
  Geom. \textbf{74} (2006), no.~1, 69--92. \MR{2260928}

\bibitem[RUY97]{RossmanUmeharaYamadaIrreducible}
Wayne Rossman, Masaaki Umehara, and Kotaro Yamada, \emph{Irreducible constant
  mean curvature {$1$} surfaces in hyperbolic space with positive genus},
  Tohoku Math. J. (2) \textbf{49} (1997), no.~4, 449--484. \MR{1478909}

\bibitem[RUY04]{RossmanUmeharaYamada}
\bysame, \emph{Mean curvature 1 surfaces in hyperbolic 3-space with low total
  curvature. {I}}, Hiroshima Math. J. \textbf{34} (2004), no.~1, 21--56.
  \MR{2046452}

\bibitem[Sma94]{Small94}
A.~J. Small, \emph{Surfaces of constant mean curvature {$1$} in {${\bf H}^3$}
  and algebraic curves on a quadric}, Proc. Amer. Math. Soc. \textbf{122}
  (1994), no.~4, 1211--1220. \MR{1209429}

\bibitem[UY93]{UmeharaYamada93}
Masaaki Umehara and Kotaro Yamada, \emph{Complete surfaces of constant mean
  curvature {$1$} in the hyperbolic {$3$}-space}, Ann. of Math. (2)
  \textbf{137} (1993), no.~3, 611--638. \MR{1217349}

\bibitem[UY96]{UmeharaYamada96}
\bysame, \emph{Surfaces of constant mean curvature {$c$} in {$H^3(-c^2)$} with
  prescribed hyperbolic {G}auss map}, Math. Ann. \textbf{304} (1996), no.~2,
  203--224. \MR{1371764}

\end{thebibliography}
\end{document}